\documentclass[12pt, emspublic, reqno]{amsart}

\usepackage{amsfonts}
\usepackage{amssymb, euscript, mathrsfs}

\newtheorem{theorem}{Theorem}[section]
\newtheorem{lemma}[theorem]{Lemma}

\numberwithin{equation}{section}

\def\Cs{\mathscr{C c 1234}}

\def\gnk{G_{n,k}}
\def\cgnk{\Cal G_{n,k}}
\def\rn{\bbr^n}
\def\sn{\bbs^n}
\def\sn1{\bbs^{n-1}}

\def\Cal{\mathcal}

\def\C{{\Cal C}}
\def\M{{\Cal M}}

\def\S{{\Cal S}}

\def\I{{\Cal I}}

\def\gnk{G_{n,k}}
\def\cgnk{\Cal G_{n,k}}

\def\f0{f_0}
\def\Fc0{\varphi_0}
\def\rn{\bbr^n}

\def\I_k {I_{-}^{k/2}}
\def\I+k {I_{+}^{k/2}}

\def\cd{\stackrel{*}{\C}\!{}_{m, k}^\a}
\def\sd{\stackrel{*}{\S}\!{}_{m, k}^\a}
\def\cd0{\stackrel{*}{\C}\!{}_{m, k}^\a}
\def\sd0{\stackrel{*}{\S}\!{}_{m, k}^\a}

\def\ncd0{\stackrel{*}{\Cs}\!{}_{m, k}^\a}

\def\bbe{{\Bbb E}}
\def\bbs{{\Bbb S}}

\def\bbm{{\Bbb M}}
\def\bbr{{\Bbb R}}

\def\bbh{{\Bbb H}}

\def\bbc{{\Bbb C}}
\def\bbz{{\Bbb Z}}

\def\const{{\hbox{\rm const}}}

\def\gnk{G_{n,k}}

\def\part{\partial}
\def\intl{\int\limits}

\def\a{\alpha}

\def\Del{\Delta}

\def\vp{\varphi}

\def\gam{\gamma}
\def\Lam{\Lambda}
\def\sig{\sigma}
\def\lam{\lambda}
\def\z{\zeta}

\def\const{{\hbox{\rm const}}}

\def\sgn{{\hbox{\rm sgn}}}

\def\part{\partial}
\def\intl{\int\limits}

\def\a{\alpha}

\def\sideremark#1{\ifvmode\leavevmode\fi\vadjust{\vbox to0pt{\vss
 \hbox to 0pt{\hskip\hsize\hskip1em
\vbox{\hsize2cm\tiny\raggedright\pretolerance10000
 \noindent #1\hfill}\hss}\vbox to8pt{\vfil}\vss}}}%

\newcommand{\be}{\begin{equation}}
\newcommand{\ee}{\end{equation}}
\newcommand{\bea}{\begin{eqnarray}}
\newcommand{\eea}{\end{eqnarray}}
\newcommand{\Bea}{\begin{eqnarray*}}
\newcommand{\Eea}{\end{eqnarray*}}

\begin{document}

\title[The Generalized Mader's  Inversion  Formulas]
{The Generalized Mader's  Inversion  Formulas for the Radon Transforms}

\author{ Y. A. Antipov }
\address{Department of Mathematics, Louisiana State University, Baton Rouge,
Louisiana 70803, USA}
\email{antipov@math.lsu.edu}

\author{ B. Rubin   }
\address{Department of Mathematics, Louisiana State University, Baton Rouge,
Louisiana 70803, USA}
\email{borisr@math.lsu.edu}

\thanks{ The first co-author was
 supported  by the NSF grant DMS-0707724.  The second co-author was
 supported  by the NSF grants PFUND-137 (Louisiana Board of Regents) and 
 DMS-0556157.}




\begin{abstract} In 1927 Philomena Mader derived elegant inversion formulas for the hyperplane Radon transform on $\bbr^n$. These formulas  differ from the original ones by Radon and seem to be forgotten. We generalize Mader's formulas to totally geodesic Radon transforms in any dimension on arbitrary constant curvature space. Another new interesting inversion formula for the $k$-plane transform was presented in the recent  book  ``Integral geometry and Radon transform" by S. Helgason. We extend this formula to arbitrary constant curvature space. The paper combines tools of integral geometry and complex analysis.
\end{abstract}

\maketitle

\section{Introduction}\label {hujkz}

Let  $ g(\theta,s)=\int f(x)\,dm (x)$ be
the Radon transform of a sufficiently good function $f$ on the Euclidean space  $\bbr^n$. Here, the integral is taken over a
 hyperplane $x\cdot \theta=s$, where $x\cdot \theta=x_1 \theta_1+\ldots+x_n \theta_n$ is the usual inner product,  
 $\theta$ is a point of  the unit sphere $\bbs^{n-1}\!=\!\{x\!\in\!\bbr^n:x^2_1\!+\!\ldots\!+\!x_n^2\!=\!1\}$, 
 and $dm (x)$ stands for the volume element on the hyperplane $x\cdot \theta=s$.  Mader's inversion formulas \cite {Mad} have the following elegant form:
 \bea\label {sui1}
 f(x)&=&A_0\frac{\partial^n}{\partial t^n}\left.F_0(x,t)\right|_{t=0} \qquad \mbox{\rm if $\,n$   is  even},
 \\
 \label {sui12}
 f(x)&=&A_1\frac{\partial^n}{\partial t^n}\left.F_1(x,t)\right|_{t=0} \qquad \mbox{\rm if $\,n$ is  odd},
\eea
where
$$
A_0=\frac{(-1)^{(n-2)/2}}{\pi(n-2)!\,\sigma_{n-2}}, \quad
A_1=\frac{(-1)^{(n-1)/2}}{2(n-2)!\,\sigma_{n-2}},
$$
$$
F_0(x,t)\!=\!\intl_{-\infty}^\infty \!G(x,s)\log|s\!-\!t|\,ds, \quad  F_1(x,t)\!=\!\intl_{-\infty}^\infty \!G(x,s)\,\sgn(s\!-\!t)\,ds,
$$
 $$
 G(x,s)=\frac{1}{\sigma_{n-1}}\intl_{\bbs^{n-1} }g(\theta,s+x\cdot \theta)\,d\sigma (\theta),
 $$
$d\sigma (\theta)$ is the surface element, and
 $\sigma_{n-1}$ is the
surface area of  $\bbs^{n-1}$.

An interesting feature of Mader's result is that  she  used tools of complex analysis (unlike most of other authors in the area).

In the present paper we derive more general inversion formulas, having the same nice structure as  (\ref{sui1}) and  (\ref{sui12}) and applicable to  totally geodesic Radon transforms in {\it any} dimension on {\it arbitrary} constant curvature space. We also generalize one recent inversion formula by S. Helgason, which has a close flavor and is  described  below.

Let $X$ be either the Euclidean space $\bbr^n$, the $n$-dimensional hyperbolic space $\bbh^n$, or the unit sphere
$\bbs^n$ in $\bbr^{n+1}$. We denote by   $\Xi$  the set of all $k$-dimensional totally
geodesic submanifolds of
 $ X, \, 1 \le k \le n-1$. The  totally geodesic Radon transform of a function $f$ on $X$
 is defined by
\be
  (R f)(\xi)=\intl_{\xi} f(x) \, d_\xi x, \qquad \xi \in \Xi,\ee
where
  $d_\xi x$ stands for  the corresponding canonical
  measure; see \cite{He} for details. Let $d(x,\xi)$ be the geodesic distance between $x \in X$ and $\xi
\in \Xi$. We introduce the {\it distance function} \be\label {mmi}
\rho(x,\xi)=\left\{ \!
 \begin{array} {ll} d(x,\xi) & \mbox{if $X=\bbr^n $,}\\
  \sin d(x,\xi) & \mbox{if $X=\bbs^n $,}\\
 \sinh \,d(x,\xi) & \mbox{if $X=\bbh^n $,}\\
  \end{array}
\right. \ee and the corresponding shifted dual Radon transform
\be\label {aaee}
(R^*_r\vp)(x)=\intl_{\rho(x,\xi)=r} \vp(\xi) \, d \mu (\xi), \qquad x \in X, \quad r>0.\ee
The last integral averages $\vp(\xi)$ over all $\xi \!\in \!\Xi$ satisfying $\rho(x,\xi)\!=\!r$.
After pioneering papers by P. Funk and J. Radon, it is known \cite {GGG, He, Ru02a,  Ru02b, Ru04b}
that $f$ can be explicitly reconstructed from $\vp=R f$ by applying  differentiation of order $k/2$ (if $k$ is odd this differentiation is fractional and interpreted in a suitable sense) to $R^*_r\vp$ in the $r^2$-variable. However, it was recently proved by Helgason \cite [p. 116]{He} that in the case $X=\bbr^n$ and $k$ even one can differentiate in $r$, rather than in $r^2$, and obtain the following result:
\be\label {hel}
f(x)=c_k \left [ \partial_r^k (R^*_r Rf)(x)\right ]_{r=0},
\ee
where $c_k$ is a constant depending only on $k$. The constant $c_k$ was not  evaluated in \cite {He}.

We  compute this constant explicitly and generalize (\ref{hel}) for  arbitrary constant curvature space $X$ and arbitrary $1 \le k \le n-1$.

{\bf Notation.} For the sake of simplicity,  throughout the paper, we assume $f$ to be infinitely differentiable. If $X$ is noncompact, we also assume that $f$ is rapidly decreasing together with derivatives of all orders.
The letter $c$ stands for a constant to be specified in every occasion; $\sigma_{n-1} =  2\pi^{n/2} \big/ \Gamma (n/2)$ is the
surface area of the unit sphere $\bbs^{n-1}$ in $\bbr^n$.

\section{Main Results}\label {Inv}
Let $\rho= \rho(x, \xi)$ be the distance function  (\ref{mmi}), $1 \le k \le n-1$. For $r>0$  we denote
\be\label {las8}
 (L^*_r \vp)(x)=\intl_{\Xi}
  \vp (\xi)\, \rho^{k+1-n} \,\sgn (\rho-r) \,d\xi, \ee
\be
 (\tilde L^*_r \vp)(x)=\intl_{\Xi}
  \vp (\xi)\, \rho^{k+1-n} \,\log |\rho^2-r^2| \,d\xi.\ee

  \begin{theorem} \label {koom} Let $\vp=R f$.  \hfill

 \noindent {\rm (i)} If $k$ is even, then
   \be\label {hel2m}
\partial_r^{k+1} (L^*_r \vp)(x)\Big |_{r=0}= d_X\,f(x),
\ee
 where
 \[d_X=\left\{ \!
 \begin{array} {ll} 2(-1)^{(k+2)/2}\sig_{n-k-1} \sig_{k-1}\,(k-1)!  & \mbox{if $\;X=\bbr^n, \bbh^n $,}\\
2\, \sigma_{n-k-1}\, \sigma_k\, \sigma_{k-1}\,(k-1)! /\sigma_n & \mbox{if $\;X=\bbs^n $.}\\
   \end{array}
\right.\]

 \noindent {\rm (ii)} If $k$ is odd, then
   \be\label {hel2}
\partial_r^{k+1} (\tilde L^*_r \vp)(x)\Big |_{r=0}=\tilde d_X\,f(x),
\ee
 where
 \[\tilde d_X=\left\{ \!
 \begin{array} {ll} \pi\,(-1)^{(k-1)/2} \sig_{n-k-1}\, \sig_{k-1} \,(k-1)! & \mbox{if $\;X=\bbr^n, \bbh^n $,}\\
2\pi (-1)^{(k-1)/2}\, \sigma_{n-k-1}\, \sigma_k\, \sigma_{k-1}\,(k-1)! /\sigma_n & \mbox{if $\;X=\bbs^n $.}\\
   \end{array}
\right.\]
\end {theorem}

  \begin{theorem} \label {koom2} Let $\vp=R f$.  If $k$ is even, then
  \be\label {hel1}
  \partial_r^k \lam_X (r)(R^*_r \vp)(x)\Big |_{r=0}=c_X\,f(x),
\ee
where
 \[
\lam_X (r)=\left\{ \!
 \begin{array} {ll} 1 & \mbox{if $\;X=\bbr^n $,}\\
  (1-r^2)^{(k-1)/2} & \mbox{if $\;X=\bbs^n $,}\\
 (1+r^2)^{(k-1)/2} & \mbox{if $\;X=\bbh^n $,}\\
  \end{array}
\right. \]

\[c_X=\left\{ \!
 \begin{array} {ll} (-1)^{k/2}(k-1)!\, \sig_{k-1} & \mbox{if $\;X=\bbr^n, \bbh^n $,}\\
  2(-1)^{k/2}(k-1)!\, \sig_{k-1} & \mbox{if $\;X=\bbs^n $.}\\
   \end{array}
\right.\]
\end {theorem}

Theorem \ref{koom} includes Mader's result for $X=\bbr^n$, $k=n-1$.
Theorem \ref{koom2} gives precise form to the afore-mentioned result by S. Helgason  and extends it to any constant curvature space $X$.

  The following lemma plays the key role in our consideration.
\begin{lemma}\label {lemm}  Let $k$ be an odd positive integer,
\be \label {pooi} \psi_k (u)=\intl_0^1 (1-v^2)^{k/2 -1} \log |u^2-v^2|\, dv.\ee
Then
\be \label {juh}\psi_k (u)=\left\{ \!
 \begin{array} {ll} P_{k-1} (u) & \mbox{if $\;0<u<1$,}\\
 P_{k-1} (u)+\pi (-1)^{(k-1)/2} \Theta (u) & \mbox{if $\;u>1$,} \\
  \end{array}
\right. \ee
where  $P_{k-1} (u)$ is a polynomial of degree $k-1$, $\Theta (u)=\int_1^u (v^2 -1)^{k/2 -1}\, dv$.
\end{lemma}
This statement is presented in Mader's paper \cite{Mad} in a slightly different form. We obtain (\ref {juh}) as a consequence of a more general result, that might be of independent interest.

Sections 3,4,5 contain the proof of Theorems \ref{koom} and \ref{koom2} for the cases $X=\bbr^n, \bbs^n$, and $\bbh^n$, respectively. Lemma \ref{lemm} is proved in Section 6.

\section{The case $X=\bbr^n$}\label {sec2}

We recall basic definitions. Let $\gnk$ ($1\le k\le n-1$) be the Grassmann manifold of $k$-dimensional linear subspaces $\z$ of $\rn$ and let
$\Xi=\cgnk$  be the
affine Grassmann manifold of all non-oriented $k$-planes $\xi$ in
$\rn$.  Each $k$-plane $\xi$  is parameterized by the pair
$(\zeta, u)$, where $\zeta \in \gnk$ and $ u \in \zeta^\perp$ (the
orthogonal complement to $\zeta $ in $\rn$).
 The manifold  $\cgnk$ will be endowed with the product measure $d\xi=d\zeta du$,
where $d\z$ is the
 $SO(n)$-invariant measure  on $\gnk$ of  total mass
$1$, and $du$ denotes the usual volume element on $\zeta^\perp$.

The $k$-plane transform  $(R f)(\xi)$ of a function $f$ on  $X=\bbr^n$ is  defined by  \be(R f)(\xi)  = \intl_\zeta f(u+v)
dv, \qquad \xi=(\zeta, u) \in \cgnk. \ee
Throughout this section we assume that
 $f$ is a rapidly decreasing infinitely differentiable function.

  The shifted dual
$k$-plane transform is defined by
\be\label {ufu}
(R^*_r \vp) (x) \! = \! \intl_{SO(n)} \!  \vp (\gam \zeta_0 \! + \! x \! + \! r\gam e_n) \,
 d\gam, \ee
 where $\zeta_0$ is an arbitrary fixed $k$-dimensional subspace of $\bbr^n$, and $e_n=(0,\ldots,0,1)$.
  This integral  averages $\vp$ over all $k$-planes  at distance $r$ from $x$. The case $r\!=\!0$ corresponds to the usual dual
Radon transform \cite{He}.

The spherical mean  of a locally integrable function $f$ at a point $x\in \rn$ is defined by
\be
(\M_t f)(x)=\frac{1}{\sig_{n-1}}\intl_{\bbs^{n-1}} f(x+t\theta )
d \,\theta, \qquad t>0, \ee
so that $\lim\limits_{t\to 0}(\M_t f)(x)=f(x)$.

\begin{lemma}  The following equalities hold: \be\label {eeqq1}
(R^*_r Rf) (x)=\sig_{k-1}\intl_r^\infty (\M_t f)(x) (t^2
-r^2)^{k/2 -1} t \, dt; \ee
\be \label {eeqq2}\intl_{\cgnk}
  \vp (\xi)\, a(\rho(x, \xi)) \, d\xi= \sig_{n-k-1} \intl_0^\infty
r^{n-k-1}\, a(r) \,(R^*_r \vp) (x) \, dr. \ee It is assumed  that $f, \; \vp$ and $a$ are measurable functions, for which
either side of the corresponding equality exists in the Lebesgue
sense.
\end{lemma}
Equality (\ref{eeqq1}) can be found in \cite [p. 115, formula (8)] {He}. Both equalities were proved in \cite [formulas (5.2), (2.29)]{Ru04b}.

\noindent {\bf Proof of Theorem \ref {koom}   (the case $X=\bbr^n$).}  We apply
 (\ref{eeqq2}) to $\vp=Rf$ and replace  $R^*_r \vp= R^*_r Rf$ according to (\ref{eeqq1}). Changing the order of integration, we  obtain the following expression:
\be\label {trf}
\sig_{n-k-1} \sig_{k-1} \intl_0^\infty (\M_t f)(x)\, t^k\, dt \intl_0^1  (tv)^{n-k-1} a(tv)(1-v^2)^{k/2 -1} dv.
   \ee
If $k$ is even, we choose $a(\cdot)$ in the form $a(\rho)=\rho^{k+1-n} \,\sgn (\rho^2-r^2)$. Then (\ref{trf}) yields
\be\label {trfm}
(L^*_r \vp)(x)=
\sig_{n-k-1} \sig_{k-1} \intl_0^\infty (\M_t f)(x)\, t^k\, \psi (r/t)\,dt, \ee
\be\label{vcb}
\psi (u)=\intl_0^1  \sgn (v-u) \,(1-v^2)^{k/2 -1} dv.
 \ee
Setting
$$
c_k=\intl_0^1 (1-v^2)^{k/2 -1} dv, \qquad \Theta (u)=\intl_1^u (v^2 -1)^{k/2 -1}\, dv,
$$
we have
\be\label{vcb1}
\psi (u)=\left\{ \!
 \begin{array} {ll} -c_k +2(-1)^{k/2}\,\Theta (u) & \mbox{if $\;0<u<1 $,}\\
  -c_k  & \mbox{if $\;u>1$.}\\
   \end{array}
\right.
\ee
This allows us to represent $\sig^{-1}_{n-k-1} \sig^{-1}_{k-1}(L^*_r \vp)(x)$ as
\bea
 &&-c_k  \intl_0^r (\M_t f)(x)\, t^k\,dt  +\intl_r^\infty (\M_t f)(x)\, t^k\,[-c_k +2(-1)^{k/2}\Theta (r/t)]\,dt\nonumber\\
 &&=2(-1)^{(k+2)/2} \intl_0^r (\M_t f)(x)\, t^k \Theta (r/t)\,dt\nonumber\\&&+\intl_0^\infty (\M_t f)(x)\, t^k\,[-c_k +2(-1)^{k/2}\Theta (r/t)]\,dt.\nonumber\eea
Since $\int_0^\infty $ is a polynomial of degree $k-1$, then
\bea
&&\partial_r^{k}(L^*_r \vp)(x)=2\sig_{n-k-1} \sig_{k-1}(-1)^{(k+2)/2}\partial_r^{k}\intl_0^r (\M_t f)(x)\, t^k \Theta (r/t)\,dt\nonumber\\
 &&=2\sig_{n-k-1} \sig_{k-1}(-1)^{(k+2)/2}\partial_r^{k-1}\intl_0^r (\M_t f)(x)\, t^k\,[\partial_r \Theta (r/t)]\,dt\nonumber\\
 &&=2\sig_{n-k-1} \sig_{k-1}(-1)^{(k+2)/2}\partial_r^{k-1}(\Lam_r f)(x),\nonumber\eea
\be\label {yyt}
(\Lam_r f)(x)=\intl_0^r (\M_t f)(x)\, (r^2-t^2)^{k/2 -1} t\,dt.\ee
We write the last integral  as $\Lam_1 +  \Lam_2$, where
$$
\Lam_1=f(x)\intl_0^r  (r^2-t^2)^{k/2 -1} t \, dt=\frac{ r^k}{k}\, f(x),
$$
$$
\Lam_2=\intl_0^r  (r^2-t^2)^{k/2 -1} [(\M_t f)(x) - f(x)]\,  t \, dt=r^k h(r),$$
$$
h(r)= \intl_0^1  (1-t^2)^{k/2 -1} [(\M_{rt} f)(x) - f(x)]\,  t \, dt.
$$
This gives
\be\label {xdc}
\lim\limits_{r\to 0}  \partial_r^k (\Lam_r f)(x)=(k-1)!  \, f(x),\ee
and therefore,
\be
\lim\limits_{r\to 0} \partial_r^{k+1}(L^*_r \vp)(x)=2(-1)^{(k+2)/2}\sig_{n-k-1} \sig_{k-1}\,(k-1)!   \, f(x),
\ee
as desired.

If $k$ is odd,  we write (\ref{trf}) with $a(\rho)=\rho^{k+1-n} \,\log |\rho^2-r^2|$
  in the form $\sig_{n-k-1} \sig_{k-1}\, (A+B(r))$, where
 $$
 A=2\intl_0^\infty (\M_t f)(x)\, t^k\,\log t\, dt \intl_0^1  (1-v^2)^{k/2 -1} dv, $$
 \be\label{reas}
 B(r)=\intl_0^\infty (\M_t f)(x)\, t^k \psi_k (r/t)\, dt,
 \ee
and $\psi_k$ being a function (\ref{pooi}). By Lemma \ref{lemm},
\bea
&&B(r)=\intl_0^r (\M_t f)(x)\, t^k \,[ P_{k-1} (r/t) +  \pi (-1)^{(k-1)/2} \Theta (r/t)] \, dt\nonumber\\
&&+\intl_r^\infty (\M_t f)(x)\, t^k \,P_{k-1} (r/t)\, dt\nonumber\\
&&= \!\intl_0^\infty  \!(\M_t f)(x)\, t^k  \,P_{k-1} (r/t) \, dt   \!+  \! \pi (-1)^{(k-1)/2}  \!\intl_0^r  \!(\M_t f)(x)\, t^k \,\Theta (r/t)] \, dt.\nonumber\eea
The integral $\int_0^\infty$ is a polynomial of $r$ of degree $k-1$. Denoting
\be \label {aaz} (\tilde L^*_r \vp)(x)=\intl_{\cgnk}
  \vp (\xi)\, \rho^{k+1-n} \,\log |\rho^2-r^2| \,d\xi,\qquad  \rho\equiv \rho(x, \xi),\ee
 and differentiating $k+1$ times, we obtain
\bea
&&
\partial_r^{k+1}(\tilde L^*_r \vp)(x)\nonumber\\
&&= \pi (-1)^{(k-1)/2}\sig_{n-k-1} \sig_{k-1}\,\partial_r^{k}\intl_0^r (\M_t f)(x)\,
 (r^2-t^2)^{k/2 -1} t \, dt.\nonumber\eea
The last integral is already known; see (\ref {yyt}). Hence, by (\ref{xdc}),
\be\label {mki}
\lim\limits_{r\to 0} \partial_r^{k+1}(\tilde L^*_r \vp)(x)=c\, f(x), \qquad c=\pi (-1)^{(k-1)/2}\sig_{n-k-1} \sig_{k-1}\,(k-1)!,  \nonumber
\ee
which completes the proof.

\vskip 0.3truecm

\noindent {\bf Proof of Theorem \ref {koom2}   (the case $X=\bbr^n$).}
 For $\vp=Rf$, equality  (\ref{eeqq1}) yields
 $(R^*_r \vp) (x)=\sig_{k-1}\, [(-1)^{k/2}\,(\Lam_r f)(x)+(\Del_r f)(x)]$, where
$(\Lam_r f)(x)$ is a function (\ref{yyt}) and
$$
(\Del_r f)(x)=\intl_0^\infty \!(\M_t f)(x) (t^2\!
-\!r^2)^{k/2 -1} t \, dt.$$
Since $(\Del_r f)(x)$ is a polynomial in the $r$-variable of degree $k-2$, owing to (\ref{xdc}), we get
 $$
\lim\limits_{r\to 0}\partial_r^k (R^*_r \vp)(x)=(-1)^{k/2}(k-1)!\, \sig_{k-1}  \, f(x).$$

\section{The case $X=\bbs^n$}

We recall that $\bbs^n$ is the unit sphere in $\bbr^{n+1}$, $d (\cdot, \cdot)$ denotes the geodesic distance on $\bbs^n$, and $\Xi$  stands for  the set of all $k$-dimensional totally
geodesic submanifolds $\xi$ of $\bbs^n$. Each $\xi\in \Xi$ is an intersection of  $\bbs^n$  with the relevant $(k+1)$-plane through the origin. Thus $\Xi$ can be identified with the Grassmann manifold $G_{n+1, k+1}$.
 The  totally geodesic Radon transform $(R f)(\xi)$ of a function $f$ on $\bbs^n$
 is defined by
\be
  (R f)(\xi)=\intl_{\xi} f(x) \, d_\xi x, \qquad \xi \in \Xi,\ee
where  $d_\xi x$ stands for  the usual Lebesgue  measure on $\xi$. Throughout this section we assume that
 $f$ is an even $C^\infty$ function.

We set $\;\bbr^{n+1}=\bbr^{k+1} \oplus \bbr^{n-k},  \quad
 \bbr^{k+1} = \bbr e_1 \oplus \ldots \oplus \bbr  e_{k+1},$\break $
\bbr^{n-k} = \bbr e_{k+2} \oplus \ldots \oplus \bbr e_{n+1},  \quad $
$ e_i $ being the coordinate unit vectors.

 Let $\xi_0=\bbs^k$ be the unit sphere in $\bbr^{k+1}$.
 Given $x \in \bbs^n, \, \xi \in \Xi,$ we denote by $r_x$ an
 arbitrary rotation satisfying $r_x e_{n+1}=x$  and set $\vp_x(\xi)=\vp(r_x \xi)$.

 Let $K=SO(n)$ be the group of rotations about the $x_{n+1}$ axis. For $\theta \in [0, \pi/2]$,
let $g_\theta$ be the rotation in the plane
$(e_{k+1}, e_{n+1})$ with
the matrix $\left[\begin{array} {cc} \sin\theta &\cos\theta\\ -\cos\theta
&\sin \theta
 \end{array} \right]$. The shifted dual
Radon transform of a function $\vp$ on $\Xi$ is defined by
\be\label {ufus}
(R^*_{ \theta} \vp )(x)=\intl_{K} \vp_x (\rho
g_\theta^{-1} \xi_0) \,d\rho.\ee
The case $\theta=0$ corresponds to the usual dual
Radon transform \cite{He}.

A geometric meaning of $(R^*_\theta\vp)(x)$ is as follows.
 Let
 \[x_\theta = g_\theta e_{n+1}=e_{k+1} \cos \theta + e_{n+1}\sin \theta \]
  so that $d(x_\theta, \xi_0)=
\theta$.  Hence,  \be
(R^*_\theta\varphi)(x) = \intl_{d(x, \xi) = \theta} \varphi(\xi) \,
d\mu(\xi)=\intl_{\rho(x,\xi)=r} \vp(\xi) \, d \mu (\xi)\equiv (R^*_r\varphi)(x),\ee
where $r=\sin \theta$, $\rho(x,\xi)=\sin d(x,\xi)$,  $d\mu(\xi)$ is the relevant normalized measure; cf. (\ref{aaee}).

We  need one more averaging operator. Given $x \in \bbs^n$ and
$ t \in (-1, 1)$, denote
 \be\label {2.21}
 ({\bbm}_t f)(x) = {(1-t^2)^{(1-n)/2}\over \sigma_{n-1}}
 \intl_{\{y \in \bbs^n:\; x \cdot y=t\}} f(y) \,d\sigma  (y).\ee
  The integral (\ref{2.21})  is the
mean value of
$f$ on the planar section of $\bbs^n$ by the hyperplane $x \cdot y=t$, and
$ \; d\sigma  (y)$
 stands for the induced Lebesgue measure on this section. Thus, $\lim\limits_{t\to 1}({\bbm}_t f)(x)=f(x)$.

 \begin{lemma}  Let $\rho(x, \xi)=\sin[d(x,\xi)]$.
 The following equalities hold: \be\label {eeqq1s}
(R^*_\theta Rf) (x)=2\sig_{k-1}\int_0^1 (1-\tau^2)^{k/2 -1} ({\bbm}_
{\tau \cos \,\theta} f)(x)\, d\tau; \ee
\bea && \label {eeqq2s} \intl_{\Xi} a(\rho(x, \xi))  \vp (\xi) \,d\xi \\
&&= \frac{\sigma_{n-k-1} \sigma_k}{\sigma_n}
 \intl^1_0(1-r^2)^{(k-1)/2}
r^{n-k-1} a(r) (R^*_{\sin^{-1} r} \vp) (x)\, dr.\nonumber\eea
It is assumed  that $f, \; \vp$ and $a$ are measurable functions, for which
either side of the corresponding equality exists in the Lebesgue
sense.
\end{lemma}
These equalities were proved in  \cite [pp. 485, 480]{Ru02b}.

\noindent {\bf Proof of Theorem \ref {koom}   (the case $X=\bbs^n$).}  Let $\vp=R f$.  We write (\ref{eeqq1s}) in the form
$$
(R^*_\theta \vp)(x)=\frac{2\sig_{k-1}}{\rho^{k-1}}\int_0^\rho (\rho^2-s^2)^{k/2 -1} ({\bbm}_
{s} f)(x)\, ds, \quad \rho=\cos\,\theta \in (0,1).$$
Setting  $r=\sin\,\theta =\sqrt {1-\rho^2}, \; s=\sqrt {1-t^2}$, and denoting
$$
(\tilde {\bbm}_t f)(x)
=(1-t^2)^{-1/2}({\bbm}_
{\sqrt {1-t^2}} f)(x),$$
so that $\lim\limits_{t \to 0}(\tilde {\bbm}_t f)(x)=f(x)$, we obtain
\be
\label {pli}(R^*_\theta \vp)(x)=\frac{2\sig_{k-1}}{(1-r^2)^{(k-1)/2}}\int_r^1 (\tilde {\bbm}_t f)(x)(t^2-r^2)^{k/2 -1} t\, dt.\ee
Then we apply
 (\ref{eeqq2s}) to $\vp=Rf$ and replace  $R^*_{\sin^{-1} r} \vp$ according to (\ref{pli}). Changing the order of integration, we  obtain
\bea\label {trfs}&&\intl_{\Xi} a(\rho(x, \xi))  \vp (\xi) \,d\xi \\
&&=
c\, \intl^1_0 (\tilde {\bbm}_t f)(x) t^k\, dt\intl_0^1  (tv)^{n-k-1} a(tv)(1-v^2)^{k/2 -1} dv,\nonumber\eea
$$
c=\frac{2\, \sigma_{n-k-1}\, \sigma_k\, \sigma_{k-1}}{\sigma_n}.
$$
This mimics (\ref{trf}) and we continue as in Section \ref {sec2}.

If $k$ is even we
choose $a(\cdot)$ in the form $a(\eta)=\eta^{k+1-n} \sgn (\eta^2 - r^2)$.
Then for $\vp=Rf$,  (\ref{trfs}) yields
\be\label {trfm}
(L^*_r \vp)(x)=
c \intl^1_0 (\tilde {\bbm}_t f)(x)\, t^k\, \psi (r/t)\,dt, \ee
where $\psi$ is a function from the previous section; see (\ref{vcb}), (\ref{vcb1}). Hence, as in Section 3, $(L^*_r \vp)(x)/c$ can be written as
\bea
 &&-c_k  \intl_0^r (\tilde {\bbm}_t f)(x)\, t^k\,dt  +\intl_r^1(\tilde {\bbm}_t f)(x)\, t^k\,[-c_k +2(-1)^{k/2}\Theta (r/t)]\,dt\nonumber\\
 &&=2(-1)^{(k+2)/2} \intl_0^r (\tilde {\bbm}_t f)(x)\, t^k \Theta (r/t)\,dt\nonumber\\&&+\intl_0^1 (\tilde {\bbm}_t f)(x)\, t^k\,[-c_k +2(-1)^{k/2}\Theta (r/t)]\,dt.\nonumber\eea
Since $\int_0^1 (\ldots )$ is a polynomial of degree $k-1$, then verbatim  of the corresponding reasoning from Section 3 yields
\[
\lim\limits_{r\to 0} \partial_r^{k+1}(L^*_r \vp)(x)= c\,(k-1)!  \, f(x)=\frac{2\, \sigma_{n-k-1}\, \sigma_k\, \sigma_{k-1}}{\sigma_n}\,(k-1)!  \, f(x).
\]

 If $k$ is odd, we choose $a(\cdot)$ in the form $a(\eta)=\eta^{k+1-n} \log |\eta^2 - r^2|$ and write (\ref{trfs}) as $c\, (A+B(r))$, where
 $$
 A=2\intl_0^1 (\tilde {\bbm}_t f)(x)\, t^k\,\log t\, dt \intl_0^1  (1-v^2)^{k/2 -1} dv=\const, $$
 $$
 B(r)=\intl_0^1 (\tilde {\bbm}_t f)(x)\, t^k \psi_k (r/t)\, dt,
 $$
and $\psi_k$ being a function (\ref{pooi}); cf. (\ref{reas}). By Lemma \ref{lemm},
\bea
&&B(r)=\intl_0^r (\tilde {\bbm}_t f)(x)\, t^k \,[ P_{k-1} (r/t) +  \pi (-1)^{(k-1)/2} \Theta (r/t)] \, dt\nonumber\\
&&+\intl_r^1 (\tilde {\bbm}_t f)(x)\, t^k \,P_{k-1} (r/t)\, dt\nonumber\\
&&= \!\intl_0^1  \!(\tilde {\bbm}_t f)(x)\, t^k  \,P_{k-1} (r/t) \, dt   \!+  \! \pi (-1)^{(k-1)/2}  \!\intl_0^r  \!(\tilde {\bbm}_t f)(x)\, t^k \,\Theta (r/t) \, dt.\nonumber\eea
The integral $\int_0^1$ is a polynomial of $r$ of degree $k-1$.
Denoting
\be \label {aazs} (\tilde L^*_r \vp)(x)=\intl_{\Xi}
  \vp (\xi)\, \rho^{k+1-n} \,\log |\rho^2-r^2| \,d\xi,\qquad  \rho\equiv \rho(x, \xi),\ee
(cf. (\ref{aaz})) and differentiating $k+1$ times, we obtain
$$
\partial_r^{k+1}(\tilde L^*_r \vp)(x)= c\,\pi (-1)^{(k-1)/2}\,\partial_r^{k}\intl_0^r (\tilde {\bbm}_t f)(x)\,
 (r^2-t^2)^{k/2 -1} t \, dt.$$
As in Section \ref{sec2}, this gives
\be
\lim\limits_{r\to 0} \partial_r^{k+1}(\tilde L^*_r \vp)(x)=c\, f(x),
\ee
$$
c=\frac{2\pi (-1)^{(k-1)/2}\, \sigma_{n-k-1}\, \sigma_k\, \sigma_{k-1}\,(k-1)! }{\sigma_n} \;(=\tilde d_X).
$$

\vskip 0.3truecm

\noindent {\bf Proof of Theorem \ref {koom2}   (the case $X=\bbs^n$).}
 Let $\vp=Rf$ and write  (\ref{pli}) in the form
\[
\label {pliw} R^*_{\sin^{-1} r} \vp=\frac{2\sig_{k-1}}{(1-r^2)^{(k-1)/2}}
\, [I_1 (r) + (-1)^{k/2}\,I_2 (r)],\] where
$$
I_1 (r)\!=\!\!\intl_0^1\!(\tilde {\bbm}_t f)(x)(t^2-r^2)^{k/2 -1} t\, dt, \quad I_2 (r)\!=\!\!\intl_0^r \! (\tilde {\bbm}_t f)(x) (r^2 -t^2)^{k/2 -1} t\, dt,$$
$I_1 (r)$ being a polynomial of degree $k-2$.
Hence, as in (\ref{xdc}), we obtain
\[
\lim\limits_{r\to 0} \partial_r^k [(1-r^2)^{(k-1)/2}\, (R^*_{\sin^{-1} r} \vp)(x)]=2(-1)^{k/2}\sig_{k-1}\,(k-1)!\,f(x).\]

\section{The case $X=\bbh^n$}

 Let $E^{n, 1}, \, n \ge 2$, be the pseudo-Euclidean space of points $x = (x_1, \dots, x_{n+1}) \in \bbr^{n+1}$ with the inner product
\be \label {ijji}
[x,y]=-x_1 y_1 - \ldots - x_ny_n + x_{n+1} y_{n+1}.\ee
We realize the $n$-dimensional hyperbolic space $X=\bbh^n$ as the upper sheet of the two-sheeted hyperboloid
$$
\bbh^n = \{ x \in \bbe^{n, 1}: [x, x] = 1, x_{n+1} > 0\}.
$$
Let $\Xi$ be the set of all $k$-dimensional totally geodesic submanifolds $\xi \subset \bbh^n, \; 1 \le k\le n-1$.
 As usual, $e_1,\dots,e_{n+1}$ denote  the coordinate unit vectors. We set  $\bbr^{n+1}= \bbr^{n-k}\oplus \bbr^{k+1}$, where
 $$
\bbr^{n-k}=  \bbr e_1 \oplus \ldots \oplus \bbr e_{n-k}, \qquad \bbr^{k+1}=\bbr e_{n-k+1} \oplus \ldots \oplus  \bbr e_{n+1},$$
 and identify $\bbr^{k+1}$ with the pseudo-Euclidean space  $E^{k, 1}$.

  In the
following $x_0 =(0,\dots,0,1)$ and $\xi_0 = \bbh^n\cap\bbe^{k,1}= \bbh^k$ denote the origins in $X$ and $\Xi$ respectively; $G=SO_0(n,1)$ is the identity component of the pseudo-orthogonal group $O(n,1)$ preserving the bilinear form (\ref{ijji}); $K=SO(n)$ and $H=SO(n-k) \times SO_0(k, 1)$ are the isotropy subgroups of $x_0$ and $\xi_0$, so that $X = G/K, \; \Xi=G/H$. One can write $f(x) \equiv f(gK), \; \varphi(\xi)\equiv\varphi(gH),\; g\in G$.  The geodesic distance between points $x$ and $y$ in $X$ is defined by $d(x,y) = \cosh^{-1}[x,y]$.

 Given $x \in X, \xi \in \Xi$, we denote by $r_x, r_\xi (\in G)$ arbitrary pseudo-rotations such that $r_x x_0 = x, \; r_\xi \xi_0 = \xi$, and write $$f_\xi (x) = f(r_\xi x), \qquad \vp_x(\xi) = \vp(r_x\xi).$$

The  totally geodesic Radon transform $(R f)(\xi)$ of a rapidly decreasing $C^\infty$ function $f$ on $\bbh^n$
 is defined by
\be
  (R f)(\xi)=\intl_{\xi} f(x) \, d_\xi x=\intl_{SO_0(k,1)} f_\xi(\gamma x_0)d\gamma, \qquad \xi \in \Xi.\ee
 As is \cite{BR}, let
$$
g_\theta\,=\,\left[\begin{array} {ccc} \cosh\theta &0 &\sinh\theta\\ 0 &I_{n-1} &0\\ \sinh\theta &0 &\cosh\theta \end{array} \right],
$$
 where $I_{n-1}$ is the unit matrix of dimension $n-1$.
The shifted dual
Radon transform of a function $\vp$ on $\Xi$ is defined by
\be\label {zss}(R^*_\theta\vp)(x)\,=\,\intl_K\vp_x(\gam g^{-1}_\theta\xi_0)\,d\gam.\ee The case $\theta=0$ corresponds to the usual dual
Radon transform.

A geometric meaning of $(R^*_\theta\vp)(x)$ is as follows.
 Let $d(x,\xi)$ be the geodesic distance between $x \in X$ and $\xi \in
 \Xi$. We set
 \[x_\theta = g_\theta e_{n+1}=
 e_1\sinh \theta + e_{n+1}\cosh \theta \] so that
$ d(x_\theta, \xi_0)=d(x_\theta,e_{n+1})=\theta$.
 Then
\be\label {54s}
d(x, r_x\gam g^{-1}_\theta \xi_0) = \theta \ee  for all $\gam\in K$ and $x\in X$. Thus,  \be
(R^*_\theta\varphi)(x) = \intl_{d(x, \xi) = \theta} \varphi(\xi) \,
d\mu(\xi)=\intl_{\rho(x,\xi)=r} \vp(\xi) \, d \mu (\xi)\equiv (R^*_r\varphi)(x),\ee
where $r=\sinh \theta$, $\rho(x,\xi)=\sinh \,d(x,\xi)$,  $d\mu(\xi)$ is the relevant normalized measure; cf. (\ref{aaee}).
 If $\varphi\in L^1(\Xi)$, then
\be\label {vfv}
\intl_\Xi\varphi(\xi)\,d\xi=\sigma_{n-k-1}\intl^\infty_0 (R^*_\theta
\varphi)(x)\,d\nu(\theta)\qquad \forall x \in X,\ee
$$d\nu(\theta)=(\sinh\theta)^{n-k-1}(\cosh\theta)^k \, d\theta;$$
see \cite[formula (2.30)]{BR}.  Given $x \in \bbh^n$ and
$ t > 1$, let
 \be\label {2.21h}
 (M_t f)(x) = {(t^2-1)^{(1-n)/2}\over \sigma_{n-1}}
 \intl_{\{y \in \bbh^n:\; [x, y]=t\}} f(y) d\sigma  (y)\ee
 be the spherical mean of
$f$ on $X=\bbh^n$, where $ \; d\sigma  (y)$
 stands for the induced Lebesgue measure. Then $\lim\limits_{t\to 1} (M_t f)(x)=f(x)$.

 \begin{lemma}   Let $\rho(x, \xi)=\sinh\,[d(x,\xi)]$.
 The following equalities hold: \be\label {eeqq1h}
(R^*_\theta Rf) (x)=\frac{\sig_{k-1}}{\tau^{k-1}}\intl_\tau^\infty (M_s f)(x) (s^2-\tau^2)^{k/2 -1} \,
 ds, \qquad \tau=\cosh\theta; \ee
\bea \label {eeqq2h} &&\intl_{\Xi}
  a(\rho(x, \xi)) \,\vp(\xi)\,  d\xi\\&&= \sig_{n-k-1} \intl^\infty_0(1+r^2)^{(k-1)/2}
r^{n-k-1} a(r) (R^*_{\sinh^{-1}r} \vp) (x)\, dr.\nonumber\eea
 It is assumed  that $f, \; \vp$ and $a$ are measurable functions, for which
either side of the corresponding equality exists in the Lebesgue
sense.
\end{lemma}
\begin{proof}
Equality (\ref{eeqq1h}) can be found in \cite {He}; see also \cite [formula (2.3)]{Ru02a}.
Equality (\ref{eeqq2h}) follows from (\ref{vfv}), (\ref{54s}), and (\ref{zss}). Indeed,
$$
I\equiv\intl_{\Xi}
  a(\rho(x, \xi)) \,\vp(\xi)\,  d\xi= \intl_{\Xi} \tilde \vp (\xi)\,  d\xi, \qquad \tilde \vp (\xi)=
  a(\rho(e_{n+1}, \xi)) \,\vp_x(\xi).$$
  Hence,
\bea
 I&=& \sigma_{n-k-1}\intl^\infty_0 (R^*_\theta
\tilde\varphi)(x)\,d\nu(\theta)\nonumber\\
&=& \sigma_{n-k-1}\intl_0^\infty d\nu(\theta) \intl_K a(\sinh d(e_{n+1}, \gam g^{-1}_\theta \xi_0))\, \vp_x(\gam g^{-1}_\theta \xi_0))\,d\gam \nonumber\\
&=& \sigma_{n-k-1}\intl_0^\infty a(\sinh \theta) \,(R^*_\theta \vp) (x) \, d\nu(\theta).\nonumber\eea
This gives (\ref{eeqq2h}).
\end{proof}

\noindent {\bf Proof of Theorem \ref {koom}   (the case $X=\bbh^n$).}  Let $\vp=R f$. Setting  $r=\sinh\,\theta =\sqrt {\tau^2-1}, \; s=\sqrt {1+t^2}$, and denoting
\be\label {bbx3}
(\tilde M_t f)(x) =(1+t^2)^{-1/2}(M_{\sqrt {1+t^2}} f)(x)
,\ee
so that $\lim\limits_{t \to 0}(\tilde M_t f)(x)=f(x)$, we obtain
\be
\label {plih} (R^*_{\sinh^{-1} r} \vp)(x)=\frac{\sig_{k-1}}{(1+r^2)^{(k-1)/2}}\int_r^\infty  (\tilde M_t f)(x)(t^2-r^2)^{k/2 -1} t\, dt.\ee
Then we apply
 (\ref{eeqq2h}) to $\vp=Rf$ and replace  $R^*_{\sinh^{-1} r} \vp$ according to (\ref{plih}).  Changing the order of integration, we  obtain
\bea\label {trfsh}&&\intl_{\Xi} a(\rho(x, \xi))  \vp (\xi) \,d\xi \\
&&=
\sigma_{n-k-1}\,  \sigma_{k-1} \intl^\infty_0 (\tilde M_t f)(x) \,t^k\, dt\intl_0^1  (tv)^{n-k-1} a(tv)(1-v^2)^{k/2 -1} dv.\nonumber\eea
This mimics (\ref{trf}) and we proceed as in the previous sections. Specifically, choose $a(\cdot)$ in the form
$a(\eta)=\eta^{k+1-n} \,\sgn (\eta^2-r^2)$. Then for $\vp=Rf$,  (\ref{trfsh}) yields
\be\label {trfmh}
(L^*_r \vp)(x)=
\sig_{n-k-1} \sig_{k-1} \intl_0^\infty (\tilde M_t f)(x)\, t^k\, \psi (r/t)\,dt; \ee
cf. (\ref{trfm}).
 Hence, as in Section 3,
\be
\lim\limits_{r\to 0} \partial_r^{k+1}(L^*_r \vp)(x)=2(-1)^{(k+2)/2}\sig_{n-k-1} \sig_{k-1}\,(k-1)!  \, f(x).
\ee

If $k$ is odd, we  choose  $a(\eta)=\eta^{k+1-n} \log |\eta^2 - r^2|$ and set
\be \label {aazh} (L^*_r \vp)(x)=\intl_{\Xi}
  \vp (\xi)\, \rho^{k+1-n} \,\log |\rho^2-r^2| \,d\xi,\qquad  \rho\equiv \rho(x, \xi),\ee
as in  (\ref{aazs}). This gives
\be
\lim\limits_{r\to 0} \partial_r^{k+1}(L^*_r \vp)(x)\!=\!c\, f(x),  \quad c\!=\!\pi (-1)^{(k-1)/2}\, \sigma_{n-k-1}\,\sigma_{k-1}\,(k-1)! .\nonumber
\ee

\vskip 0.3truecm

\noindent {\bf Proof of Theorem \ref {koom2}   (the case $X=\bbh^n$).}
 For $\vp=Rf$, equality  (\ref{eeqq1h}) yields
\be
 (R^*_{\sinh^{-1}r} \vp) (x)=\frac{\sig_{k-1}}{(1+r^2)^{(k-1)/2}}\, [(-1)^{k/2}\,(\Lam_r f)(x)+(\Del_r f)(x)],
\ee
where (cf. (\ref{bbx3}), (\ref{yyt}))
$$
(\Lam_r f)(x)=\intl_0^r (\tilde M_t f)(x)\, (r^2-t^2)^{k/2 -1} t\,dt,
$$$$
(\Del_r f)(x)=\intl_0^\infty \!(\tilde M_t f)(x) (t^2\!
-\!r^2)^{k/2 -1} t \, dt.$$
Hence, as in (\ref{xdc}), we obtain
\[
\lim\limits_{r\to 0} \partial_r^k [(1+r^2)^{(k-1)/2}\, (R^*_{\sinh^{-1} r} \vp)(x)]=(-1)^{k/2}\sig_{k-1}\,(k-1)!\,f(x).\]

\section{Proof of Lemma \ref{lemm}}\label {lal}

The statement will be obtained as a consequence of the following more general result, which  generalizes the reasoning from \cite{Mad}.
\begin{lemma}\label {ook} Let $m\in \bbz$, $-1<\a<m+1$, $\a\notin \bbz$, $0<u<\infty$. Then
\bea\label {A.1}
\phi(u)&=&\intl_{-1}^1(1+\xi)^\a(1-\xi)^{m-\a}\log|\xi-u|\,d\xi\\
\label {A.2}
&=&\mu_\a(u)\,\Theta_\a(u)+P_{m+1}(u),
\eea
where
$$
\mu_\a(u)=\left\{
\begin{array}{ll}-\pi\,\cot \a\pi, &  \mbox{if $0<u<1$,}\\
(-1)^m \pi \,\csc \a\pi, & \mbox{if $1<u<\infty$,}\\
  \end{array}
\right. $$
\be\label {A.3}
\Theta_\a(u)=\intl_1^u(1+\xi)^\a|1-\xi|^{m-\a}d\xi,
\ee
and  $P_{m+1}(u)$ is a  polynomial of degree $m+1$ defined  by
\be\label {A.4}
P_{m+1}(u)=\phi(1)-(-1)^m \pi \,\csc \a\pi\,\sum_{r=1}^{m+1}  \lam_r\, (u^r -1),\ee
$$
\lam_r= \frac{1}{r}\sum_{\ell=1}^{m+1-r} (-1)^\ell {m-\a\choose \ell } {\a\choose m+1-r-\ell}.
$$
\end{lemma}

\begin{proof} The result will be obtained using the classical  tools of complex analysis. Consider  auxiliary functions
$$
\zeta (z)=(z+1)^\a(z-1)^{m-\a}, \qquad h(z,u)=\zeta (z)\, \log\frac {z-u}{z-1}; \qquad z\in \bbc.$$
To fix a single branch of  $\zeta (z)$, we cut the $z$-plane along the segment $-1\le\xi\le 1$ and choose $
\arg(z\pm 1)\in[0,2\pi].$ Then
$\zeta (z)\sim z^m,$ $z\to\infty$, and
\be\label {A.11}
\zeta (\xi\pm i0)=\left\{
\begin{array}{ll} (-1)^me^{\mp i\a\pi}\,|1+\xi|^\a|1-\xi|^{m-\a},  &  \mbox{if $-1<\xi<1$,}\\
{}\\
|1+\xi|^\a|1-\xi|^{m-\a},  & \mbox{if $1<\xi<\infty$.}\\
  \end{array}
\right.\ee

{\bf Case 1.} Let $1<u<\infty$. Continue the cut $[-1,1]$ along the real axis up to the point $u$
and fix  a branch of the logarithmic function so that $\arg(z-u)\in[0,2\pi]$ and,
as before, $\arg(z-1)\in[0,2\pi]$.  Then
$$
\log\frac {z-u}{z-1}\sim 0, \quad z\to\infty,
$$
and
\be\label {A.12}
\log\frac {\xi\pm i0-u}{\xi\pm i0-1}=\log\left|\frac {\xi-u}{\xi-1}\right|\pm
\left\{\begin{array}{cc}
i\pi, & \; 1<\xi<u,\\
0, & \; 0<\xi<1.\\
\end{array}\right.
\ee
At the point $z=\infty$, the function $h(z,u)$ has an integer order, and its residue is
\be\label {A.13}
-c_{-1}=\mathop{\rm res}\limits_{z=\infty}h(z,u)=\frac {1}{2\pi i}\intl\limits_{C^-} h(z,u)dz.
\ee
Here $c_{-1}$ is the coefficient at $z^{-1}$ in the Laurent  series of $h(z,u)$
 in a neighborhood of the  infinite point, and $C^-$ is a circle of  radius $R$ centered at $z=0$.  It is assumed that $R$ is big enough, so that $z=\infty$ is the only singular point outside of $C^-$. The direction of integration on $C^-$
is  chosen clockwise.
Deforming the contour
$C^-$, we obtain
\be\label {A.14}
\mathop{\rm res}\limits_{z=\infty}h(\xi,u)=\frac {1}{2\pi i}\int\limits_L h(\xi,u)\,d\xi,
\ee
where $L=[u,1]^-\cup [1,-1]^-\cup [-1,1]^+\cup[1,u]^+$, and $[a,b]^+$ and $[a,b]^-$ are the upper
and lower sides of a cut $[a,b]$, respectively.
Owing to (\ref{A.11}) and (\ref{A.12}), (\ref{A.14}) yields
\be\label {A.15}
\mathop{\rm res}\limits_{z=\infty}h(z,u)=\Theta_\a (u)-\frac {(-1)^m\sin\a\pi}{\pi}\int_{-1}^1
(1+\xi)^\a(1-\xi)^{m-\a}\log\left|\frac {\xi-u}{\xi-1}
\right|\,d\xi.\nonumber\ee
In other words,
\be\label {kuku}
\phi(u)=\phi(1)+\mu_\a(u)\,\Theta_\a(u)-\mu_\a(u)\, \mathop{\rm res}\limits_{z=\infty}h(z,u),\ee
and it remains to evaluate the residue on the right-hand side.
 We set
\be\label {A.16}
h(z,u)=-W(z,u)+W(z,1),
\ee
where
\bea\label {A.17}
W(z,u)&=&-z^m\left(1+\frac {1}{z}\right)^{\a}\left(1-\frac {1}{z}\right)^{m-\a}\log\left(1-\frac {u}{z}\right)\nonumber\\
&=&\sum_{p=0}^\infty a_p \, z^{m-p}\,
\sum_{q=0}^\infty b_q\, z^{-q}\,
\sum_{r=1}^\infty c_r \,z^{-r},
\nonumber\eea
$$
a_p =  {\a\choose p}, \qquad b_q=  (-1)^q \,  {m-\a\choose q}, \qquad
c_r=  \frac{u^r}{r}.
   $$
A simple calculation gives $W(z,u)=\gam z^{-1} +$  terms of degree $\neq -1$, with
\be
\gam=\sum_{r=1}^{m+1} \lam_r\, u^r, \qquad \lam_r= \frac{1}{r}\sum_{\ell=1}^{m+1-r} (-1)^\ell {m-\a\choose \ell } {\a\choose m+1-r-\ell}.\nonumber\ee
Hence,
\be\label {ppp}
\mathop{\rm res}\limits_{z=\infty}h(z,u) = \sum_{r=1}^{m+1}  \lam_r\, (u^r -1)= Q_{m+1}(u),\ee
 and, by (\ref{kuku}), we finally get $\phi(u)=\mu_\a(u)\,\Theta_\a(u)+P_{m+1}(u)$, where
 \be\label {99o} P_{m+1}(u)=\phi(1)-\mu_\a(u)\,Q_{m+1}(u)=\phi(1)-\frac {(-1)^m \pi}{\sin \a\pi} \,Q_{m+1}(u).\ee

\noindent {\bf Case 2.} Let  $0<u<1$. Then the substitute for (\ref{A.14})  has the form
\be\label {A.14s}
\mathop{\rm res}\limits_{z=\infty}h(z,u)=\frac {1}{2\pi i}\int\limits_{L'} h(\xi,u)\,d\xi,
\ee
where $L'= [1,u]^-\cup [u,-1]^-\cup[-1,u]^+\cup [u,1]^+$, and the residue of  $h(z,u)$  at the infinite point remains unchanged.
 Now, (\ref{A.3}), (\ref{A.11}), the counterpart of formula (\ref{A.12}) for this case
\be\label {A.14d}
\log\frac {\xi\pm i0-u}{\xi\pm i0-1}=\log\left|\frac {\xi-u}{\xi-1}\right|\pm
\left\{\begin{array}{cc}
i\pi, & \; u<\xi<1,\\
0, & \; 0<\xi<u,\\
\end{array}\right.
\ee
   and (\ref{ppp})  yield
\bea
 Q_{m+1}(u)&=&(-1)^{m-1}\cos \,\a\pi\,\Theta_\a (u) \nonumber\\&-&\frac {(-1)^m\sin\a\pi}{\pi}\int_{-1}^1
(1+\xi)^\a(1-\xi)^{m-\a}\log\left|\frac {\xi-u}{\xi-1}
\right|\,d\xi.
 \nonumber\eea
This gives
\bea
\phi(u)&=&\phi(1)-\pi\cot \a\pi \,\Theta_\a(u)-\frac {(-1)^m \pi}{\sin \a\pi}\, Q_{m+1}(u)\nonumber\\
&=&\mu_\a (u)\,\Theta_\a(u)+P_{m+1}(u),\nonumber\eea
as desired.
\end{proof}

Let us turn to the particular case $\a=m-\a=\frac {k}{2}-1$, where $k$ is an odd  positive integer.  Then
\be \phi (u)=\intl_{-1}^1(1-\xi^2)^{k/2 -1}\log|\xi-u|\,d\xi,\ee
and Lemma \ref {ook}  yields
\be
\phi (u)=P_{k-1}(u)+\left\{
\begin{array}{ll} 0, &  \mbox{if $0<u<1$,}\\
\pi(-1)^{(k-1)/2}\Theta (u), & \mbox{if $1<u<\infty$,}\\
  \end{array}
\right. \ee
where $P_{k-1}(u)$ is a  polynomial of degree $k-1$, and
$$\Theta (u)=\intl_1^u (\xi^2-1)^{k/2 -1}\,d\xi.
$$
Since
 $$\phi (u)=\Big (\intl_{-1}^0+\intl_0^1 \Big ) (\ldots)= \intl_0^1(1-\xi^2)^{k/2 -1}\log|\xi^2-u^2|\,d\xi,$$
we arrive at Lemma \ref{lemm}.

\end{document}